\documentclass[11pt, leqno]{amsart}
\usepackage{amsthm,amsmath,amssymb,amscd,verbatim}
\usepackage[curve,matrix,arrow,frame,tips]{xy}
\usepackage{graphicx}
\usepackage{hyperref}

\theoremstyle{plain}
\newtheorem{theorem}{Theorem}[section]
\newtheorem{lemma}[theorem]{Lemma}
\newtheorem{cor}[theorem]{Corollary}
\newtheorem{prop}[theorem]{Proposition}

\theoremstyle{remark}

\theoremstyle{definition}

\newtheorem{defn}[theorem]{Definition}
\newtheorem{Remark}[theorem]{Remark}

\numberwithin{equation}{section}
\numberwithin{theorem}{section}

\newcommand{\ZZ}{\mathbb Z}

\newcommand{\RR}{\mathbb R}

\newcommand\Adm{{\rm Adm}}

\newcommand\hfld[2]{\smash{\mathop{\hbox to 10mm{\rightarrowfill}}
     \limits^{\scriptstyle#1}_{\scriptstyle#2}}}
\newcommand\hflg[2]{\smash{\mathop{\hbox to 10mm{\leftarrowfill}}
     \limits^{\scriptstyle#1}_{\scriptstyle#2}}}

\usepackage{color}

\title{Vertexwise criteria for admissibility of alcoves}

\author{Thomas J. Haines}
\address{University of Maryland\\
Department of Mathematics\\
College Park, MD 20742-4015 U.S.A.}
\email{tjh@math.umd.edu}
\author{Xuhua He}
\address{University of Maryland\\
Department of Mathematics\\
College Park, MD 20742-4015 U.S.A.\\ and department of mathematics, HKUST, Hong Kong}
\email{xuhuahe@math.umd.edu}

\begin{document}

\thanks{Research of T.H.~partially supported by NSF DMS-1406787 and by the MSRI Fall 2014 program on {\em Geometric Representation Theory}.}

\date{}

\maketitle

\begin{abstract} We give a new description of the set ${\rm Adm}(\mu)$ of admissible alcoves as an intersection of certain ``obtuse cones'' of alcoves, and we show this description may be given by imposing conditions vertexwise. We use this to prove the vertexwise admissibility conjecture of Pappas-Rapoport-Smithling \cite[Conj.~4.39]{PRS}. The same idea gives simple proofs of two ingredients used in the proof of the Kottwitz-Rapoport conjecture on existence of crystals with additional structure \cite{He}.
\end{abstract}

\bigskip
 


\section{Introduction}

\subsection{Background}
Let $\Sigma = (X^*, X_*, R, R^\vee, \Pi)$ be a based reduced root system with simple roots $\Pi$ and with extended affine Weyl group $\widetilde{W}$. Let $\mu \in X_*$. In \cite{KR}, Kottwitz and Rapoport introduced two finite subsets, ${\rm Adm}(\mu)$ and ${\rm Perm}(\mu)$, of $\widetilde{W}$.  Roughly speaking $\widetilde{W}$ can be identified with the set of alcoves in the apartment $V = X_* \otimes \mathbb R$. The set ${\rm Adm}(\mu)$ is then defined by using the Bruhat order on the set of alcoves, and ${\rm Perm}(\mu)$ is defined by imposing certain vertexwise conditions on alcoves (see Definition \ref{KR_defn}). Originally, these sets were introduced because of their expected relation to special fibers of Rapoport-Zink local models \cite{RZ}. Namely, if $(\Sigma, \mu)$ comes from the data $(G, \{\mu \})$ defining a local model $M^{\rm loc}$ with Iwahori-level structure, then ${\rm Perm}(\mu)$ was expected to parametrize a natural stratification of the special fiber of $M^{\rm loc}$ and ${\rm Adm}(\mu)$ was expected to parametrize those strata in the closure of the generic fiber of $M^{\rm loc}$. The set ${\rm Adm}(\mu)$ also arises naturally as the support of the Bernstein function $z_\mu$ in the center of the affine Hecke algebra with parameters attached to $\Sigma$ (cf.~\cite[Prop.~14.3.1]{HR2}).

The inclusion ${\rm Adm}(\mu) \subseteq {\rm Perm}(\mu)$ was proved in \cite{KR}, and Kottwitz and Rapoport raised in \cite{KR} the question of when equality holds. In the cases where the equality ${\rm Adm}(\mu) = {\rm Perm}(\mu)$ could be proved, it has been used to establish the topological flatness of $M^{\rm loc}$ (cf.~G\"{o}rtz \cite{Go1, Go2, Go3} and Smithling \cite{Sm2, Sm3, Sm4, Sm1}). Recently Pappas and Zhu \cite{PZ} defined group-theoretic ``local models'' attached to any pair $(G, \{ \mu \})$ where $G$ is a tamely ramified group over a $p$-adic field, and showed that the strata in the special fiber are parametrized by the $\{\mu\}$-admissible set. 

The equality ${\rm Adm}(\mu) = {\rm Perm}(\mu)$ in general is too much to hope for: although it holds for all $\mu$ if $\Sigma$ has type A, it fails for $\Sigma$ irreducible of rank $\geq 4$ and not of type A, for all sufficiently regular $\mu$ \cite{HN}. But this does not cover all the cases where $\mu$ is minuscule (of interest for most local models) or where $\mu$ is a sum of minuscule cocharacters.  In fact one could hope (see \cite[after Thm.~3.4]{Ra05}) that ${\rm Adm}(\mu) = {\rm Perm}(\mu)$ whenever $\mu$ is minuscule. If $\Sigma$ only involves types A, B, C, or D, this was proved by Kottwitz-Rapoport \cite{KR} (Types A and C) and Smithling \cite{Sm2, Sm1} (Types B and D). However, the article of Lisa Sauermann \cite{Sau} shows that the equality ${\rm Adm}(\mu) = {\rm Perm}(\mu)$ fails for the minuscule cocharacters attached to the exceptional types ${\rm E}_6$ and ${\rm E}_7$.


The situation vis-\`{a}-vis ${\rm Perm}(\mu)$ having been clarified, the goal of this article is to prove that nevertheless ${\rm Adm}(\mu)$ can always be described by vertexwise conditions, and to use this to prove the vertexwise admissibility conjecture of Pappas-Rapoport-Smithling (\cite[Conjecture 4.2]{PR} and \cite[Conjecture 4.39]{PRS}). 

\subsection{The notion of strong admissibility}
The starting point is a new description of ${\rm Adm}(\mu)$ given in terms of ``obtuse cones'' of alcoves. Let $\mathcal C$ be the dominant Weyl chamber in $V = X_* \otimes \mathbb R$ with closure $cl(\mathcal C)$, and let $\bar{\mathcal C}$ be the antidominant Weyl chamber. Let $\bar{\bf a}$ be the base alcove in $\bar{\mathcal C}$. Fix a Weyl group element $w \in W$ and $\mu \in X_* \cap cl(\mathcal C)$ with $t_\mu \in W_{\rm aff}\tau_\mu$ (see $\S\ref{notation_sec}$). 
\begin{defn} \label{B_alc_defn}
Let $B(t_{w\mu}(\bar{\bf a}), w)$ be the set of $x \in W_{\rm aff}\tau_\mu$ such that there is a sequence of affine roots $\tilde{\alpha}_1, \dots, \tilde{\alpha}_r$ and corresponding reflections $s_{\tilde{\alpha}_1}, \dots, s_{\tilde{\alpha}_r}$ with
\begin{enumerate}
\item[(1)] $s_{\tilde{\alpha}_r} \cdots s_{\tilde{\alpha}_1} t_{w\mu}(\bar{\bf a}) = x(\bar{\bf a})$;
\item[(2)] Every reflection is in the $w$-{\em opposite direction}, that is, for each $i$, $s_{\tilde{\alpha}_{i-1}} \cdots s_{\tilde{\alpha}_1} t_{w\mu}(\bar{\bf a})$ and ${\bf b}$ are on opposite sides of the affine root hyperplane $H_{\tilde{\alpha}_i}$, where ${\bf b}$ is any sufficiently regular alcove ${\bf b} \subset w\bar{\mathcal C}$.
\end{enumerate}
Equivalently, $B(t_{w\mu}(\bar{\bf a}), w)$ is the set of $ x \in W_{\rm aff}\tau_\mu$ such that $x(\bar{\bf a}) \leq_{\bf b} t_{w\mu}(\bar{\bf a})$ for all sufficiently regular ${\bf b} \subset w\bar{\mathcal C}$. Here $\leq_{\bf b}$ denotes the Bruhat order on the set of alcoves determined by ${\bf b}$. 
\end{defn}
This set $B(t_{w\mu}(\bar{\bf a}), w)$ corresponds to an ``obtuse cone'' of alcoves. We then define the new set of {\em strongly-admissible} elements
$$
{\rm Adm}^{st}(\mu) = \bigcap_{w \in W} B(t_{w\mu}(\bar{\bf a}), w).
$$
This is reminiscent of the set ${\rm Perm}^{st}(\mu)$ of {\em strongly-permissible} elements introduced in \cite{HN}, where similar conditions were imposed vertexwise (see Definition \ref{J_defs}). The key result of \cite{HN} is that the inclusion ${\rm Perm}^{st}(\mu) \subseteq {\rm Adm}(\mu)$ always holds.  Clearly, we also have ${\rm Adm}^{st}(\mu) \subseteq {\rm Perm}^{st}(\mu)$. Thus we have
\begin{equation*} 
{\rm Adm}^{st}(\mu) \subseteq {\rm Perm}^{st}(\mu) \subseteq {\rm Adm}(\mu).
\end{equation*}
Our main result, Theorem \ref{adm^st-J=adm}, asserts that these inclusions are equalities, and that in fact stronger parahoric equalities hold (which are needed to prove the vertexwise admissibility conjecture). In particular, as ${\rm Adm}(\mu)$ coincides with ${\rm Perm}^{st}(\mu)$, it is defined by vertexwise conditions. 

\subsection{Statement of main theorem} 

In order to state the theorem we need to define parahoric analogues of the above notions. Let ${\rm Vert}(\bar{\bf a})$ denote the set of vertices for $\bar{\bf a}$. For $J \subseteq S_{\rm aff}$ any subset of simple affine reflections, let $W_J$ be the subgroup of $W_{\rm aff}$ generated by the $s_{\tilde{\alpha}} \in J$.  Given a vertex $a \in {\rm Vert}(\bar{\bf a})$, we define its {\em type} to be the set of simple reflections $s_{\tilde{\alpha}} \in S_{\rm aff}$ which do not fix $a$; similarly, we say $a$ has {\em type not in $J$} if $a$ is fixed by every $s_{\tilde{\alpha}}$ in $J$.

\begin{defn} \label{J_defs} For any subset $J \subseteq S_{\rm aff}$, define 
\begin{enumerate}
\item[(1)] $\Adm^J(\mu) = W_J {\rm Adm}(\mu) W_J$.
\item[(2)] \label{perm_st-J_def} ${\rm Perm}^{st, J}(\mu)$ is the set of $x \in W_{\rm aff} \tau_\mu$ such that, for every $a \in {\rm Vert}(\bar{\bf a})$ of type not in $J$, and for every $w \in W$, we have $x(a) \in B(t_{w\mu}(a), w)$.
\end{enumerate}
\end{defn}
Here following \cite[Def.~4.1]{HN}, we define $B(t_{w\mu}(a), w)$ to be the set of points we get by applying a sequence of reflections in the $w$-opposite direction to the point $t_{w\mu}(a)$ (comp.~Definition \ref{B_alc_defn} with $\bar{\bf a}$ replaced by $a$).

\begin{theorem}\label{adm^st-J=adm}
Let $\mu$ be a dominant cocharacter, and let $J \subseteq S_{\rm aff}$ be a subset. Then we have the equalities $${\rm Adm}^{st}(\mu) W_J={\rm Perm}^{st, J}(\mu)= {\rm Adm}^J(\mu).$$ In particular, taking $J = \emptyset,$ we obtain ${\rm Adm}^{st}(\mu) = {\rm Perm}^{st}(\mu) = {\rm Adm}(\mu)$.
\end{theorem}

The ``acute cone'' of alcoves $\mathcal C(\bar{\bf a}, w)$ introduced in \cite{HN} and analogous to $w\mathcal C$ also plays an important part in this paper, mainly via Corollary \ref{cone_cor}: if $x \in {\rm Adm}(\mu) \cap \mathcal C(\bar{\bf a}, w)$, then $x \leq t_{w\mu}$.

\subsection{Outline of proof}
The proof of Theorem \ref{adm^st-J=adm} is based on the following inclusions: 

\begin{itemize}
\item ${\rm Adm}^{st}(\mu)W_J \subseteq {\rm Perm}^{st,J}(\mu)$. This follows directly from the definitions.
\item ${\rm Perm}^{st, J}(\mu) \subseteq \Adm(\mu) W_J$. See Proposition \ref{key_prop}.

\noindent The special case ${\rm Perm}^{st}(\mu) \subseteq \Adm(\mu)$ was proved in \cite{HN}. Here we follow a similar strategy to prove this inclusion. 
\item $\Adm(\mu) \subseteq \Adm^{st}(\mu)$. See Proposition \ref{adm in adm^st}.

\noindent The proof relies on a detailed study of sequences of reflections in a given direction. 
\item $\Adm^J(\mu)=\Adm(\mu)W_J$. See Proposition \ref{adm-J}. 

\noindent This identity is equivalent to the compatibility property of admissible sets: $$\Adm^J(\mu) \cap \widetilde{W}^J=\Adm(\mu) \cap \widetilde{W}^J.$$ This was first proved in \cite[Theorem~6.1]{He}. Here we give a different proof using acute cones and sequences of reflections in a given direction. 
\end{itemize}

Theorem \ref{adm^st-J=adm} results by putting all these inclusions together.

\subsection{Applications of main theorem}
As a consequence of Theorem \ref{adm^st-J=adm}, we get another proof (in $\S\ref{Zhu_sec}$) of the additivity of admissible sets. 

\begin{theorem} \label{Zhu} \textup{(}\cite[Theorem~5.1]{He}\textup{)}
Let $\mu$ and $\nu$ be any dominant cocharacters, and let $J \subseteq S_{\rm aff}$. Then $${\rm Adm}^J(\mu + \nu) = {\rm Adm}^J(\mu) \, {\rm Adm}^J(\nu).$$
\end{theorem}

The reference \cite{He} handles explicitly the special case $\Adm(\mu+\nu)=\Adm(\mu) \, \Adm(\nu)$. However, the proof there works in the general case.



The additivity property and compatibility property of admissible sets played an important role in the second author's proof \cite{He} of the Kottwitz-Rapoport conjecture on non-emptiness of certain unions of affine Deligne-Lusztig varieties. 


As another consequence of Theorem \ref{adm^st-J=adm}, we prove the Pappas-Rapoport-Smithling conjecture (\cite[Conjecture 4.2]{PR} and \cite[Conjecture 4.39]{PRS}) on vertexwise admissibility. See section \ref{locmod_sec} for an explanation of the notation used here.

\begin{theorem}\label{PRS_conj} Let $G,T,\widetilde{W}(G), S_{\rm aff}(\Sigma), \{\mu\}, \bar{\bf a}$ be as in section \ref{locmod_sec}. 
For any subset $J \subseteq S_{\rm aff}(\Sigma)$, let $\bar{\bf a}_J$ be the corresponding facet in the closure of $\bar{\bf a}$, having set of vertices ${\rm Vert}(\bar{\bf a}_J)$, and let ${\rm Adm}^J(\{ \mu \}) \subset \widetilde{W}(G)$ be the associated $\{ \mu \}$-admissible subset. For $a \in {\rm Vert}(\bar{\bf a})$, let $J_a \subset S_{\rm aff}(\Sigma)$ be the subset such that $W_{J_a}$ is the subgroup of $\widetilde{W}$ fixing $a$. Then $$\Adm^J(\{ \mu \})=\bigcap_{a  \in {\rm Vert}(\bar{\bf a}_J)} {\rm Adm}^{J_a}(\{\mu \}).$$
\end{theorem}

See \cite[$\S4$]{PR} and \cite[$\S4.5$]{PRS} (and also our Remark \ref{final_rem}) for discussions of this conjecture and its relation to local models of Shimura varieties. We prove Theorem \ref{PRS_conj} in a preliminary form in Theorem \ref{PRS_prelimconj} and in the desired form in Theorem \ref{PRS_conjfinal}.

\section{Notation} \label{notation_sec}

\subsection{} Let $\Sigma = (X^*,X_*, R, R^\vee, \Pi)$ be a based reduced root system and let $\langle ~ , ~ \rangle : X^* \times X_* \rightarrow \ZZ$ be the natural perfect pairing between the free abelian groups $X^*$ and $X_*$. Let $V=X_* \otimes \RR$. For any $\alpha \in R$, we have a reflection $s_\alpha$ on $V$ defined by $s_{\alpha}(x) = x - \langle \alpha , x \rangle \alpha^\vee$. For any $(\alpha, k) \in R \times \ZZ =: \widetilde{R}$, we have an affine root $\tilde{\alpha} = \alpha + k$ and an affine reflection
$s_{\tilde{\alpha}}$ on $V$ defined by $s_{\tilde{\alpha}}(x) = x - (\langle \alpha , x \rangle+k)\alpha^\vee$. 

The finite Weyl group $W$ is the subgroup of ${\rm GL}(V)$ generated by $s_\alpha$ for $\alpha \in R$ and the affine Weyl group $W_{\rm aff}$ is the subgroup of ${\rm Aut}(V)$ generated by $s_{\tilde{\alpha}}$ for affine roots $\tilde{\alpha} = \alpha+ k$. Let $Q^\vee$ be the subgroup of $X_*$ generated by the coroots $\alpha^\vee$. We may identify $W_{\rm aff}$ with $Q^\vee \rtimes W$ in natural way. Inside ${\rm Aut}(V)$ define the {\em extended affine Weyl group} $\widetilde{W} := X_* \rtimes W$, where $X_*$ acts by translations on $V$. For $\lambda \in X_*$, the symbol $t_\lambda \in \widetilde{W}$ stands for the translation $v \mapsto v + \lambda$ on $V$.

Let $H_{\tilde{\alpha}}$ denote the hyperplane (wall) in $V$ fixed by the reflection $s_{\tilde{\alpha}}$.
~The connected components of the set
$V - \bigcup_{\alpha \in R} H_\alpha$ will be called {\em chambers} and the connected components of the set $V - \bigcup_{(\alpha,k) \in \widetilde R} H_{\alpha + k}$ will be called {\em alcoves}. It is known that $W$ acts simply transitively on the set of chambers and $W_{\rm aff}$ acts simply transitively on the set of alcoves. 

For any subset $\mathcal S$ of the apartment $V$, let $cl(\mathcal S)$ denote its closure.

\subsection{} The set of simple roots $\Pi \subset R$ determines the set $R_+$ of positive roots. The group $W$ is generated by $S=\{s_\alpha ~  |~  \alpha \in \Pi\}$ as a finite Coxeter group. The dominant chamber is the set 
$$\mathcal C = \{ x \in V \,\, | \,\, \langle \alpha , x \rangle > 0, \,\,\, \mbox{for every}  
\,\,\, \alpha \in \Pi \}.$$ 
The base alcove is defined to be the set
$$\bar{\bf a}=\{x \in V ~ | ~ -1<\langle \alpha , x \rangle<0, \,\,\, \mbox{for every} \,\,\, \alpha \in R_+\}.$$

Let $S_{\rm aff}$ be the set of $s_{\tilde{\alpha}}$, where $H_{\tilde{\alpha}}$ is a wall of $\bar {\bf a}$. Then $S \subset S_{\rm aff}$ and $W_{\rm aff}$ is generated by $S_{\rm aff}$ as a Coxeter group. 

The extended affine Weyl group $\widetilde W= X_* \rtimes W$ also acts on $V$ and on the set of alcoves. Let $\Omega$ be the isotropy group in $\widetilde W$ of the base alcove $\bar{\bf a}$. It is known that $\widetilde W=W_{\rm aff}\rtimes \Omega$. For $w,w'\in W_{\rm aff}$,
$\tau,\tau'\in\Omega$, we say $w\tau \leq w'\tau'$ if and only if $\tau=\tau'$ and $w\leq w'$ (with respect to the Bruhat order on the Coxeter group $W_{\rm aff}$). We put $\ell(w\tau)=\ell(w)$. 

By abuse, we will call an element $\mu \in X_*$ a {\em cocharacter}. Given $\mu \in X_*$, define $\tau_\mu \in \Omega$ by the requirement $t_\mu \in W_{\rm aff} \tau_\mu$. 

\subsection{} The following definitions are due to Kottwitz and Rapoport \cite{KR}.
\begin{defn} \label{KR_defn} 
Let $\mu \in X_* \cap cl(\mathcal C)$. define:
\begin{enumerate}
\item[(a)] ${\rm Adm}(\mu) = \{ x \in W_{\rm aff}\tau_\mu ~ | ~ x \leq t_\lambda, \,\,\, \mbox{for some $\lambda \in W\mu$} \}$.
\item[(b)] ${\rm Perm}(\mu) = \{ x \in W_{\rm aff}\tau_\mu ~ | ~ x(a) - a \in {\rm Conv}(W\mu), \,\,\, \forall a \in {\rm Vert}(\bar{\bf a}) \}$.
\end{enumerate}
Here ${\rm Vert}(\bar{\bf a})$ is the set of vertices (minimal facets) for $\bar{\bf a}$, and ${\rm Conv}(W\mu)$ is the convex hull of the Weyl-orbit $W\mu \subset V$.
\end{defn}

\subsection{}
Let $J \subseteq S_{\rm aff}$ be any subset of simple affine reflections.  Recall that $W_J$ is the subgroup of $W_{\rm aff}$ generated by the $s_{\tilde{\alpha}} \in J$. We say that $W_J$ is a parahoric subgroup if $W_J$ is finite. 

Let $\widetilde{W}^J$ be the set of minimal length elements of the cosets $x W_J$. Let $\bar{\bf a}_J$ be the subset of elements in the closure of $\bar{\bf a}$ which are pointwise fixed by $W_J$. Let ${\rm Vert}(\bar{\bf a}_J)$ be the collection of vertices for $\bar{\bf a}_J$. 

\subsection{}
For any affine hyperplane $H$, let $H^+$ be the half-space containing $\bar{\bf a}$. For $w \in  W$, we let $H^{w+}$ be the side of $H$ containing all alcoves sufficiently deep inside the Weyl chamber $w\mathcal C$. (This agrees with the definition of $H^{w+}$ given in \cite[$\S5$]{HN}.) Let $H^{w-}$ be the side of $H$ opposite to $H^{w+}$.

For any facet ${\bf f}$ not contained in an affine hyperplane $H_{\tilde{\alpha}}$, we say the affine reflection ${\bf f} \mapsto s_{\tilde{\alpha}}({\bf f})$ is {\em in the $w$-direction} (resp.~$w$-{\em opposite direction}) if the generic point of ${\bf f}$ lies in $H^{w-}$ (resp.~$H^{w+}$) and the generic point of $s_{\tilde{\alpha}}({\bf f})$ lies in $H^{w+}$ (resp.~$H^{w-}$). We often say $s_{\tilde{\alpha}}$ is in the $w$-direction (or $w$-opposite direction) if ${\bf f}$ is understood.  

Following \cite[Def.~5.4]{HN}, let ${\mathcal C}(\bar{\bf a}, w)$ be the set of alcoves ${\bf c}$ such that there is a gallery $\bar{\bf a} = {\bf c}_0, {\bf c}_1, \dots, {\bf c}_l = {\bf c}$ where each reflection ${\bf c}_{i-1} \mapsto {\bf c}_i$ is in the $w$-direction.  For $x \in \widetilde{W}$, we often write $x \in \mathcal C(\bar{\bf a}, w)$ if $x(\bar{\bf a}) \in \mathcal C(\bar{\bf a}, w)$.

\subsection{}  
We will use the following standard lemma about the Bruhat order $\leq$ for a Coxeter system $(W,S)$. One reference is \cite[Prop.\,5.9]{Hum}.

\begin{lemma} \label{std_lem}
For $x, y \in W$ and $s \in S$, if $x \leq y$, then $xs \leq ys$ or $x \leq ys$\, \textup{(}or both\textup{)}.
\end{lemma} 

\section{On the parahoric strongly permissible set} \label{order_sec}


\subsection{Parabolic Bruhat order in Coxeter groups}\label{order}

Let $(W,S)$ be a Coxeter group, and $J \subseteq S$ a set of simple reflections generating the subgroup $W_J \subseteq W$. We again use $\leq $ for the Bruhat order on $(W,S)$. Recall that $W^J = \{ w \in W ~ | ~ w < wt  \,\, \forall t \in J\}$. 

\begin{lemma} \label{standard_lem} We have the following statements:
\begin{enumerate}
\item[(1)] Any element $x \in W$ can be written uniquely as $x = x^J x_J$, where $x^J \in W^J$ and $x_J \in W_J$. Furthermore, $x^J \leq x$.
\item[(2)] If $x \leq y$, then $x^J \leq y^J$.
\item[(3)] For any $x \in W$, and $s \in S$, then 
\begin{equation*} \label{(sx)^J}
(sx)^J = \begin{cases}  x^J, \,\,\,\, \mbox{if $x^{-1}sx \in W_J$} \\ s x^J, \,\,\,\, \mbox{if $x^{-1} s x \notin W_J$}. \end{cases}
\end{equation*}
Furthermore, $(sx)^J = {\rm min}\{ x^J, sx^J \}$ if $sx < x$.
\end{enumerate}
\end{lemma}

\begin{proof}
The first two statements are standard. The last one can be easily proved as follows. If $x^{-1}sx \in W_J$, clearly $(sx)^J = x^J$. Henceforth assume $x^{-1}s x\notin W_J$, so that $(sx)^J \neq x^J$. Suppose $sx < x$. Then $(sx)^J \leq x^J$. Since $(sx)^J \neq x^J$, we have $(sx)^J = (sx^J)^J < x^J$. But then $sx^J < x^J$, and hence $(sx)^J = (s x^J)^J = sx^J$, as desired. Next suppose $x < sx$. Then $x^J \leq (sx)^J$, and we must have $x^J < (sx)^J \leq sx^J$. Comparing lengths we obtain $(sx)^J = sx^J$.  The final statement follows from $(sx)^J \leq sx^J$ (which always holds) combined with $(sx)^J \leq x^J$ (which holds when $sx < x$). 
\end{proof}

\begin{lemma} \label{x'<sx}
Let $x,y \in W,$ $J \subseteq S$, and $s \in S$.  If $x^J \leq y^J$ and $sy < y$, then $(x')^J \leq (sy)^J$, where $x' = {\rm min}\{x, sx \}$.
\end{lemma}

\begin{proof}
By Lemma \ref{standard_lem}(3), $(x')^J \in \{x^J, (sx)^J \} \subseteq \{x^J, sx^J\}$, and $(x')^J = {\rm min}\{ x^J, sx^J \}$. Similarly, as $sy < y$ we have $(sy)^J = {\rm min}\{y^J, s y^J \}$.  Since $x^J \leq y^J$, we may conclude $(x')^J \leq (sy)^J$.
\end{proof}

\begin{prop} \label{Deodhar}
Let $\{J_i\}$ be a family of subsets $J_i \subseteq S$. Then for $x, y \in W,$
$$
x^{J_i} \leq y^{J_i},\,\, \forall i ~\Longleftrightarrow ~ x^{\cap_i J_i} \leq y^{\cap_i J_i}.
$$
\end{prop}

\begin{Remark} The proof is essentially in \cite[Lem.~3.6]{Deo}, which handled the case where $\cap_i J_i = \emptyset$. \end{Remark}

\begin{proof}
We may assume that $y = y^{\cap_i J_i}$. We argue by induction on $\ell(y)$. If $y = 1$, then $x \in \bigcap_i W_{J_i} = W_{\cap_i J_i}$, and so $x^{\cap_i J_i} =1$.

Now we assume that $y > 1$. Then there exists $s \in S$ with $sy < y$.  Let $x' ={\rm min}\{x, sx\}$. Then by Lemma \ref{x'<sx},
$$
(x')^{J_i }\leq (sy)^{J_i}, \,\,\ \forall i.
$$
By the induction hypothesis, $(x')^{\cap_i J_i} \leq (sy)^{\cap_i J_i} = sy$.  Since $x^{\cap_i J_i} \in \{ (x')^{\cap_i J_i}, s(x')^{\cap_i J_i} \}$ and $sy < y$, we conclude $x^{\cap_i J_i} \leq y$. \end{proof}

\subsection{Parahoric strong permissibility}

Now we return to the context of extended affine Weyl groups. In this subsection we will use some results from \cite{HN}, where in contrast to this paper, the Bruhat order is defined using a base alcove $A_0$ in the positive Weyl chamber $\mathcal C$. The reader may easily check that if one replaces $A_0$ with $\bar{\bf a}$ in the required statements from \cite{HN}, the resulting statements still hold (the same proofs work).

Our first task is to prove a variant of \cite[Prop.~6.1]{HN}.

\begin{prop} \label{HN_variant}
Let $x_1, x_2 \in W_{\rm aff}$, $w \in W$, $J \subseteq S_{\rm aff}$. If
\begin{itemize}
\item $x_1 \in W_{\rm aff}^J,$ 
\item for every $a \in {\rm Vert}(\bar{\bf a})$ of type not in $J$, $x_1(a) \in B(x_2(a), w),$ 
\item $x_1(\bar{\bf a}), x_2(\bar{\bf a}) \in \mathcal C(\bar{\bf a}, w),$
\end{itemize}
then $x_1 \leq x_2$.
\end{prop}

It is easiest to envision what this is saying when $J = \emptyset$: suppose the alcoves $x_1(\bar{\bf a}), x_2(\bar{\bf a})$ lie deep inside the Weyl chamber $w\mathcal C$ (and thus automatically in $\mathcal C(\bar{\bf a}, w)$); then a sufficient condition for $x_1 \leq x_2$ in the Bruhat order is that each vertex $x_1(a)$ of $x_1(\bar{\bf a})$ lies in the ``$w$-negative obtuse cone'' $B(x_2(a), w)$ emanating from the vertex of the same type in the alcove $x_2(\bar{\bf a})$.

\begin{proof}
By \cite[Lem.~6.4]{HN}, there exists for $i \notin J$ an alcove ${\bf a}_i$ with the property
that for all $y  \in W_{\rm aff},$ 
$$y^{-1}(\bar{\bf a}) \in \mathcal C({\bf a}_i, ww_0) \Longrightarrow  (yx_1)^{S_{\rm aff}-\{i\}} \leq (yx_2)^{S_{\rm aff}-\{i\}}.
$$
Choose $y \in W_{\rm aff}$ such that
$$
y^{-1}(\bar{\bf a}) \in \mathcal C(\bar{\bf a}, ww_0) \cap \bigcap_{i \notin J} \mathcal C({\bf a}_i, ww_0).
$$
Then $(yx_1)^{S_{\rm aff}-\{i\}} \leq (yx_2)^{S_{\rm aff}-\{i\}}, \,\,\, \forall i \notin J.$
By Proposition \ref{Deodhar}, $(yx_1)^J \leq (yx_2)^J$.  But then standard properties of the Bruhat order (cf.\,Lemma \ref{std_lem}) imply that $yx_1 \leq yx_2 u$ for some $u \in W_J$.
By \cite[Proof of Lem.~6.2]{HN}, $\ell(yx_1) = \ell(y) + \ell(x_1)$. So $x_1 \leq x_2 u$.  Since $x_1 \in W_{\rm aff}^J$, we obtain $x_1 \leq x_2$.
\end{proof}


\begin{prop} \label{key_prop} Let $J \subseteq S_{\rm aff}$. Then $${\rm Perm}^{st, J}(\mu) \subseteq\Adm(\mu) W_J.$$
\end{prop}

\begin{proof}


Let $x \in {\rm Perm}^{st, J}(\mu)$. Then by definition $x^J \in {\rm Perm}^{st, J}(\mu)$. By \cite[Cor.~5.6]{HN}, there exists $w \in W$ with $x^J(\bar{\bf a}) \in \mathcal C(\bar{\bf a}, w)$. 
By \cite[Cor.~5.7]{HN}, $t_{w\mu}(\bar{\bf a}) \in \mathcal C(\bar{\bf a}, w)$. By Proposition \ref{HN_variant}, $x^J \leq t_{w\mu}$. So $x^J \in {\rm Adm}(\mu)$ and $x \in {\rm Adm}(\mu)W_J$. 
\end{proof}

\section{On the strongly admissible set}

In this section, we prove that admissible alcoves are strongly admissible.

\begin{prop}\label{adm in adm^st}
For any dominant cocharacter $\mu$, we have $$\Adm(\mu) \subseteq \Adm^{st}(\mu).$$
\end{prop}

The following easy lemma will be used several times without comment in what follows.  We leave the proof to the reader.

\begin{lemma} Let $w \in W$ and $x = t_\lambda \eta \in X_* \rtimes W$. Let $\beta$ be an affine root and suppose for an alcove ${\bf c}$, the reflection ${\bf c} \mapsto s_\beta({\bf c})$ is in the $w$-opposite direction.  Then the reflection $x{\bf c} \mapsto (x s_\beta x^{-1})(x{\bf c})$ is in the $\eta w$-opposite direction.
\end{lemma}

Let $R^{\vee}_+$ denote the set of positive coroots.

\begin{lemma} \label{3planes}
Suppose $\alpha^\vee \in -wR^{\vee}_+$ for some $w \in W$. Let ${\bf c}$ be any alcove. Then there is a sequence of reflections in the $w$-opposite direction taking ${\bf c}$ to $t_{\alpha^\vee}{\bf c}$.
\end{lemma}

\begin{proof}
Choose any translation $\lambda_0$ in the closure of ${\bf c}$ and let $N$ be the integer such that $\lambda_0 \in H_{\alpha + N + 1}$, so that $\lambda_0 + \alpha^\vee \in H_{\alpha + N - 1}$. If $\langle \alpha, {\bf c} \rangle \subset (-N-2, -N-1)$, then $s_{\alpha + N} s_{\alpha + N+1} ({\bf c}) =t_{\alpha^\vee}{\bf c}$.  If $\langle \alpha, {\bf c} \rangle \subset (-N-1, -N)$, then $s_{\alpha + N -1}s_{\alpha + N}({\bf c}) = t_{\alpha^\vee}{\bf c}$. In either case, we have found two affine reflections, each in the $w$-opposite direction, whose product takes ${\bf c}$ to $t_{\alpha^\vee}{\bf c}$. 
\end{proof}

We can now prove Proposition \ref{adm in adm^st}. 

\begin{proof}
It is enough to prove
\begin{enumerate}
\item[(1)] For each $w \in W$, $t_{w\mu} \in {\rm Adm}^{st}(\mu)$.
\item[(2)] ${\rm Adm}^{st}(\mu)$ is closed under the Bruhat order.
\end{enumerate}

First we prove (1).  We will prove that $\{ t_\nu ~ | ~ \nu \in \Omega(\mu) \} \subset {\rm Adm}^{st}(\mu)$, where $\Omega(\mu)$ is the set of $W$-orbits of dominant cocharacters $\lambda$ with $\lambda \preceq \mu$ in the dominance partial ordering $\preceq$ on cocharacters. Fix $\nu \in \Omega(\mu)$. For each $w \in W$, we need to find a sequence of affine reflections in the $w$-opposite direction taking $t_{w\mu}(\bar{\bf a})$ to $t_{\nu}(\bar{\bf a})$. Since $\nu - w\mu$ is a sum of coroots $\alpha^\vee \in -wR^\vee_+$, it is enough to fix such an $\alpha^\vee$ and a $\lambda_0 \in \Omega(\mu)$, and show that we may get from $t_{\lambda_0}(\bar{\bf a})$ to $t_{\lambda_0 + \alpha^\vee}(\bar{\bf a})$ by a sequence of reflections in the $w$-opposite direction. This follows from Lemma \ref{3planes}.

It remains to prove (2). Suppose $x \in {\rm Adm}^{st}(\mu)$ and let $\tilde{\alpha}$ be an affine reflection taking positive values on $\bar{\bf a}$ and having vector part $\alpha \in R$, such that $s_{\tilde{\alpha}}x < x$.  It is enough to show that $y := s_{\tilde{\alpha}}x \in {\rm Adm}^{st}(\mu)$. Note that $x(\bar{\bf a})$ and $\bar{\bf a}$ are separated by the affine hyperplane $H_{\tilde{\alpha}}$, so that $\langle \tilde{\alpha}, x(\bar{\bf a}) \rangle \subset (-\infty, 0)$. 

Fix $w \in W$. We need to show that there is a sequence of reflections in the $w$-opposite direction, taking $t_{w\mu}(\bar{\bf a})$ to $y(\bar{\bf a})$. Note that $w^{-1}\alpha^\vee$ is either in $R^\vee_+$ or in $-R^\vee_+$. 

\noindent{\bf Case 1:} $\alpha^\vee \in -wR^\vee_+$. In this case the reflection taking $x(\bar{\bf a})$ to $y(\bar{\bf a})$ is already in the $w$-opposite direction, since $\langle \tilde{\alpha}, x(\bar{\bf a}) \rangle \subset (-\infty, 0)$. As $x \in {\rm Adm}^{st}(\mu)$, there is a sequence of reflections $s_\bullet$ in the $w$-opposite direction taking $t_{w\mu}(\bar{\bf a})$ to $x(\bar{\bf a})$, and the concatenation $s_{\tilde{\alpha}} \circ s_\bullet$ is the desired sequence.

\noindent{\bf Case 2:} $\alpha^\vee \in wR^\vee_+$.  Hence $\alpha^\vee \in -s_\alpha wR_+^\vee$. Write $\tilde{\alpha} = \alpha + k$, for $k \in \mathbb Z$.  Since $\alpha + k$ takes positive values on $\bar{\bf a}$, we either have (i) $k \geq 1$, or (ii) $k = 0$ and $\alpha <0$. 

As $x \in {\rm Adm}^{st}(\mu)$, there is a sequence of reflections $s_\bullet$ in the $s_\alpha w$-opposite direction, taking $t_{s_\alpha w\mu}(\bar{\bf a})$ to $x(\bar{\bf a})$. Conjugating by $s_{\tilde{\alpha}}$, we get a sequence $s'_\bullet := s_{\tilde{\alpha}} s_\bullet s_{\tilde{\alpha}}$ in the $w$-opposite direction, taking $s_{\tilde{\alpha}}t_{s_{\alpha}w\mu}(\bar{\bf a})$ to $y(\bar{\bf a})$. Note $s_{\tilde{\alpha}}t_{s_{\alpha}w\mu}(\bar{\bf a})= s_{\alpha +k'} t_{w\mu}(\bar{\bf a})$ where $k' = k - \langle \alpha, w\mu \rangle$.

We claim that the reflection $s_{\alpha +k'}$ taking $t_{w\mu}(\bar{\bf a})$ to $s_{\tilde{\alpha}}t_{s_{\alpha}w\mu}(\bar{\bf a})$ is in the direction of $-\alpha^\vee$, i.e., in the $w$-opposite direction. To see this, let $a \in \bar{\bf a}$ be arbitrary, and note that $s_{\alpha+k'}t_{w\mu}(a) = -(k + \langle \alpha, a \rangle)\alpha^\vee + t_{w\mu}(a)$, and further note that for either case (i) or (ii), $k + \langle \alpha, a \rangle > 0$. The concatenation $s'_\bullet \circ s_{\alpha + k'}$ gives the desired sequence of reflections in the $w$-opposite direction taking $t_{w\mu}(\bar{\bf a})$ to $y(\bar{\bf a})$.
\end{proof}

\begin{cor} \label{cone_cor}
Let $x \in {\rm Adm}(\mu) \cap \mathcal C(\bar{\bf a}, w)$.  Then $x \leq t_{w\mu}$.
\end{cor}

\begin{proof} 
By Proposition \ref{adm in adm^st}, $x \in \Adm^{st}(\mu) \subseteq {\rm Perm}^{st}(\mu)$. By the proof of Proposition \ref{key_prop} (alternatively, by the proof of \cite[Prop.~2]{HN}), we obtain $x \leq t_{w \mu}$. 
\end{proof}

This result exhibits a natural translation element which an admissible element necessarily precedes in the Bruhat order. It is a substantial generalization of the fact that for a dominant cocharacter $\mu$ and an element  $t_\lambda w \in Wt_\mu W$, we have $t_\lambda w \in {\rm Adm}(\mu)$ if and only if $t_\lambda w \leq t_\lambda$.  This statement was proved in \cite{H01, HP} for $\mu$ minuscule and the general case can be proved in a similar manner. It also follows immediately from \cite[Proposition 2.1]{HL}.

\begin{Remark} \label{simpleThm1_rem}
Using Proposition \ref{adm in adm^st} with the inclusions ${\rm Adm}^{st}(\mu) \subseteq {\rm Perm}^{st}(\mu) \subseteq {\rm Adm}(\mu)$ (the latter proved in \cite[Prop.~2]{HN}), we get the equality ${\rm Adm}^{st}(\mu) = {\rm Perm}^{st}(\mu) = {\rm Adm}(\mu)$ (and thus also Corollary \ref{cone_cor}) without using the material of $\S\ref{order_sec}$ and $\S\ref{adm-J_sec}$. Those sections are needed to prove the full version of Theorem \ref{adm^st-J=adm}.
\end{Remark}

\begin{Remark} \label{Conv_rem}
Let $B_0$ denote the negative obtuse cone in $V$ generated by the coroots $-\alpha^\vee$ with $\alpha \in R_+$. The following properties of the convex hull ${\rm Conv}(W\mu)$ are well-known:
\begin{enumerate}
\item[(i)] ${\rm Conv}(W\mu) = \bigcap_{w \in W} w\mu + w(B_0)$ 
\item[(ii)] ${\rm Conv}(W\mu) \cap w\mathcal C = (w\mu + w(B_0)) \cap w\mathcal C, \,\,\, \forall w \in W$.
\end{enumerate}
We may regard ${\rm Adm}(\mu)$ as the alcove-theoretic analogue of ${\rm Conv}(W \mu)$.
Theorem \ref{adm^st-J=adm} (for $J = \emptyset$) and Corollary \ref{cone_cor} imply the alcove-theoretic analogues of (i) and (ii):
\begin{enumerate}
\item[(I)] ${\rm Adm}(\mu) = \bigcap_{w \in W} B(t_{w\mu}(\bar{\bf a}), w)$
\item[(II)] ${\rm Adm}(\mu) \cap \mathcal C(\bar{\bf a}, w) = \{ x \leq t_{w\mu}\} \cap \mathcal C(\bar{\bf a}, w), \,\,\, \forall w \in W$.
\end{enumerate}
\end{Remark}


\section{On the parahoric admissible set} \label{adm-J_sec}

In this section, we prove 

\begin{prop}\label{adm-J}
Let $\mu$ be a dominant cocharacter and $J \subseteq S_{\rm aff}$. Then $$\Adm^J(\mu) \cap \widetilde{W}^J=\Adm(\mu) \cap \widetilde{W}^J.$$
\end{prop}

\subsection{}
Define  $\mathcal C(\bar{\bf a}_J) = \bigcap_{\tilde{\alpha} \in J} H^+_{\tilde{\alpha}}$. In a sense, this is the ``Weyl chamber'' containing $\bar{\bf a}$ with ``apex'' the facet $\bar{\bf a}_J$. Choose $a_J \in \bar{\bf a}_J$, and define 
$$
W(\bar{\bf a}_J) = \{ w \in W ~ | ~  w\mathcal C \subset H^+_{\tilde{\alpha}} - a_J,\,\, \forall \tilde{\alpha} \in J\}.
$$
This set is independent of the choice of $a_J$. It is straightforward to verify the equality
\begin{equation} \label{J-obtcone_eq}
cl(\mathcal C(\bar{\bf a}_J)) = \bigcup_{w \in W(\bar{\bf a}_J)} cl(a_J + w\mathcal C).
\end{equation}
Thus, our ``Weyl chamber'' is in this sense a union of certain ordinary Weyl chambers, translated by $a_J$.

\begin{lemma} \label{stub_lem}
Let $w \in W$. 
\begin{enumerate}
\item[(1)] For an alcove ${\bf b}$ in the apartment,  ${\bf b} \cap cl(a_J + w\mathcal C) \neq \emptyset \Longrightarrow {\bf b} \in \mathcal C(\bar{\bf a}, w)$. 
\item[(2)] If $w \in W(\bar{\bf a}_J)$, then the alcove $w\mu + \bar{\bf a}$ is contained in $\mathcal C(\bar{\bf a}_J)$.
\end{enumerate}
\end{lemma}

\begin{proof}
For (1), assume ${\bf b}$ meets $cl(a_J + w\mathcal C)$.  We may write
$$
b = a + w\tilde{c}
$$
for some $b \in {\bf b}$, some $\tilde{c} \in cl(\mathcal C)$, and some $a \in \bar{\bf a}$ sufficiently close to $a_J$.  So for every $\alpha \in w(R_+)$, we have
$$
\langle \alpha, a \rangle \leq \langle \alpha, b \rangle.
$$
This means that for any affine hyperplane $H$ separating $\bar{\bf a}$ from ${\bf b}$, we must have $\bar{\bf a} \subset H^{w-}$ and ${\bf b} \subset H^{w+}$. But then it is clear that any minimal gallery joining $\bar{\bf a}$ to ${\bf b}$ is in the $w$-direction, i.e., ${\bf b} \in \mathcal C(\bar{\bf a}, w)$.

For (2), let $\tilde{\alpha} \in J$ have vector part $\alpha$. We must show that $w\mu + \bar{\bf a} \subset H_{\tilde{\alpha}}^+$. If not, then $\tilde{\alpha}(w\mu + \bar{\bf a}) \subset (-\infty, 0)$. Because $\tilde{\alpha}$ is positive on $\bar{\bf a}$, this implies $\langle \alpha, w\mu \rangle \leq -1$. On the other hand, by the very definition of $W(\bar{\bf a}_J)$ we have $a_J + w\mathcal C \subset H_{\tilde{\alpha}}^+$. Thus $\tilde{\alpha}$ takes non-negative values on $w\mu + a_J \in cl(a_J + w\mathcal C)$. Since $\tilde{\alpha}$ vanishes on $a_J$, this means $\langle \alpha, w\mu \rangle \geq 0$, a contradiction.
\end{proof}

\subsection{}
Now we can prove Proposition \ref{adm-J}.

\begin{proof} 

Suppose $x \in \Adm^J(\mu) \cap \widetilde{W}^J$.  Let $x_0 \in W_J x W_J$ be the unique element of minimal length, an element which is automatically in ${\rm Adm}(\mu)$. It is easy to see that $x_0(\bar{\bf a}) \subset \mathcal C(\bar{\bf a}_J)$. By (\ref{J-obtcone_eq}), there is a $w \in W(\bar{\bf a}_J)$ such that $x_0(\bar{\bf a})$ meets $cl(a_J + w\mathcal C)$. By Lemma \ref{stub_lem}(1), $x_0(\bar{\bf a}) \in \mathcal C(\bar{\bf a},w)$. Since $x_0$ is in ${\rm Adm}(\mu)$, by Corollary \ref{cone_cor}, we have $x_0 \leq t_{w\mu}$.

Therefore there exists a sequence of affine reflections $s_{\tilde{\alpha}_1}, \dots, s_{\tilde{\alpha}_r}$ with the properties
\begin{enumerate}
\item[(i)] $s_{\tilde{\alpha}_r}\cdots s_{\tilde{\alpha}_1}x_0 = t_{w\mu}$;
\item[(ii)] for each $i\geq 1$, $\bar{\bf a}$ and $s_{\tilde{\alpha}_{i-1}} \cdots s_{\tilde{\alpha}_1} x_0(\bar{\bf a})$ are on the same side of $H_{\tilde{\alpha}_i}$.
\end{enumerate} 

We may write $x = yx_0$ for a (unique) $y \in W_J$. Write $\tilde{\beta}_i = y\tilde{\alpha}_i$ for all $i$. Then applying $y$ to (i) and (ii) we get a sequence of affine reflections $s_{\tilde{\beta}_i}$ such that
\begin{enumerate}
\item[(i')] $s_{\tilde{\beta}_r} \cdots s_{\tilde{\beta}_1} x = yt_{w\mu}$;
\item[(ii')] for each $i \geq 1$, $y\bar{\bf a}$ and $s_{\tilde{\beta}_{i-1}} \cdots s_{\tilde{\beta}_1} x(\bar{\bf a})$ are on the same side of $H_{\tilde{\beta}_i}$.
\end{enumerate}
Thus $x \leq_{y\bar{\bf a}} yt_{w\mu}$, in the Bruhat order $\leq_{y \bar{\bf a}}$ determined by the alcove $y\bar{\bf a}$. But $\bar{\bf a}$ and $y\bar{\bf a}$ are on the same side of $H_{\tilde{\beta}_i}$ for all $i$. If not, then $H_{\tilde{\beta}_i} \supset \bar{\bf a}_J$ and also $H_{\tilde{\alpha}_i} \supset \bar{\bf a}_J$. But $t_{w\mu}(\bar{\bf a}) \subset \mathcal C(\bar{\bf a}_J)$ (Lemma \ref{stub_lem}(2)) and therefore $\bar{\bf a}$ and $t_{w\mu}(\bar{\bf a})$ are on the same side of $H_{\tilde{\alpha}_i}$, which is incompatible with property (ii). 

Thus in fact $x \leq yt_{w\mu}$.  Write $yt_{w\mu} = t_{w'\mu}y$ for some $w' \in W$. Then since $x \in \widetilde{W}^J$ and $x \leq t_{w'\mu}y$, a standard property of the Bruhat order yields $x \leq t_{w'\mu}$ (use Lemma \ref{std_lem} repeatedly).
\end{proof}




\section{Proof of Theorem \ref{Zhu}} \label{Zhu_sec}

\begin{proof}
The proof for the inclusion ${\rm Adm}^J(\mu + \nu) \subseteq {\rm Adm}^J(\mu) \, {\rm Adm}^J(\nu)$ is similar to the proof in \cite[Theorem 5.1]{He}. Suppose $x \leq y_1 t_{w(\mu+\nu)} y_2 =y_1 t_{w\mu} \cdot t_{w\nu} y_2$ for some $y_1, y_2 \in W_J$. By looking at subwords of a reduced expression for $y_1 t_{w\mu}$ followed by a reduced expression of $t_{w\nu} y_2$ we may realize $x$ as a product $x = x' x''$ with $x' \leq y_1 t_{w\mu}$ and $x'' \leq t_{w\nu} y_2$. 

The converse inclusion ${\rm Adm}^J(\mu) {\rm Adm}^J(\nu) \subseteq {\rm Adm}^J(\mu+ \nu)$ can be proved as in \cite{He} following Zhu's argument using global Schubert varieties.  Here we give a simple proof based on Theorem \ref{adm^st-J=adm}.  

We first show that $\Adm(\mu) \Adm(\nu) \subseteq \Adm(\mu+\nu)$. Let $x' \in {\rm Adm}(\mu)={\rm Adm}^{st}(\mu)$ and $x'' \in {\rm Adm}(\nu)={\rm Adm}^{st}(\nu)$. Fix $w \in W$. There exists a sequence of affine reflections $s'_\bullet$ in the $w$-opposite direction taking $t_{w\mu}(\bar{\bf a})$ to $x'(\bar{\bf a})$. Write $x' = t_\lambda \eta \in X_* \rtimes W$, and set $w' = \eta^{-1}w$. Then $x'^{-1}t_{w\nu} x' = t_{w'\nu}$. Choose a sequence $s''_\bullet$ in the $w'$-opposite direction taking $t_{w'\nu}(\bar{\bf a})$ to $x''(\bar{\bf a})$.  Then the concatenation $(x' s''_\bullet x'^{-1}) \circ (t_{w\nu} s'_\bullet t_{-w\nu})$ is a sequence of reflections in the $w$-opposite direction, taking $t_{w(\mu + \nu)}(\bar{\bf a})$ to $x'x''(\bar{\bf a})$. This shows that $x'x'' \in {\rm Adm}^{st}(\mu + \nu) = {\rm Adm}(\mu + \nu)$.

By Theorem \ref{adm^st-J=adm}, $\Adm^J(\mu)=W_J \Adm(\mu)$ and $\Adm^J(\nu)=\Adm(\nu) W_J$. Hence 
$$
\Adm^J(\mu) \, \Adm^J(\nu)=W_J \Adm(\mu) \, \Adm(\nu) W_J \subseteq W_J \Adm(\mu+\nu) W_J=\Adm^J(\mu+\nu).
$$
\end{proof}

\section{Preliminary version of the vertexwise admissibility conjecture}

A preliminary version of Theorem \ref{PRS_conj} can be formulated as follows. 

\begin{theorem} \label{PRS_prelimconj}
Suppose $\mathcal J$ is a collection of subsets $J \subseteq S_{\rm aff}$, and let $K = \bigcap_{J \in \mathcal J} J$.  Then
$$
\Adm^K(\mu)=\bigcap_{J \in \mathcal J} \Adm^J(\mu).
$$
\end{theorem}

\begin{proof}
Suppose $x \in {\rm Adm}^J(\mu)$ for all $J \in \mathcal J$. By Theorem \ref{adm^st-J=adm}, we have $x \in {\rm Perm}^{st,J}(\mu)$ for all $J$. But by definition
$$
\bigcap_{J \in \mathcal J} {\rm Perm}^{st,J}(\mu) = {\rm Perm}^{st,K}(\mu)
$$
and so $x \in {\rm Perm}^{st,K}(\mu)$. Then $x \in {\rm Adm}^K(\mu)$, again by Theorem \ref{adm^st-J=adm}.
\end{proof}

\begin{Remark}
The special case where $K = \emptyset$ can be proved without using any of the material of $\S\ref{order_sec}$, as follows. Suppose $x \in {\rm Adm}^J(\mu)$ for all $J$. Then $x^J \in \widetilde{W}^J \cap {\rm Adm}(\mu)$ by  Proposition \ref{adm-J}.  Hence $x^J \in {\rm Perm}^{st}(\mu)$ by Remark \ref{simpleThm1_rem}. Thus for every $a \in {\rm Vert}(\bar{\bf a}_J)$, $x(a) = x^J(a)$ satisfies the strong permissibility condition. As this holds for every $J \in \mathcal J$ and $\cap_{J \in \mathcal J} J = \emptyset$, we deduce that $x \in {\rm Perm}^{st}(\mu)$. But then $x \in {\rm Adm}(\mu)$, again by Remark \ref{simpleThm1_rem}.
\end{Remark}

\section{The vertexwise admissibility conjecture} \label{locmod_sec}

\subsection{Notation}
We will follow the notation and set-up of \cite[$\S4.1-4.5$]{PRS}. Let $F$ be a nonarchimedean local field, and let $L$ denote the completion of the maximal unramified extension inside some separable closure $F^{\rm sep}$ of $F$. Set $I := {\rm Gal}(L^{\rm sep}/L)$. Let $G$ be a connected reductive group defined over $L$, and let $S \subset G$ be a maximal $L$-split torus with centralizer $T$; recall $T$ is a maximal torus since $G$ is quasi-split over $L$ by Steinberg's theorem. Let $W = N(L^{\rm sep})/T(L^{\rm sep})$ be the absolute Weyl group of $(G,T)$; here $N = {\rm Norm}_G(T)$.   Let $W_0 = N(L)/T(L)$ be the relative Weyl group.

Recall that associated to $G \supset T$ there is a reduced root datum $\Sigma = (X^*, X_*, R, R^\vee)$, whose affine Weyl group can be identified with the Iwahori-Weyl group over $L$ of the simply-connected group $G_{\rm sc}$: $W_{\rm aff}(\Sigma) = \widetilde{W}(G_{\rm sc})$ (see \cite[p.~195]{HR1}). The root system $\Sigma$ gives the notion of affine hyperplane and alcove in the apartment $V := X_*(T)_I \otimes \mathbb R = X_*(S) \otimes \mathbb R$ (as usual we identify the apartment for $S$ with $V$ by choosing, once and for all, a special vertex of the apartment). We also fix the notion of a positive Weyl chamber in $V$ (corresponding to a Borel subgroup $B \supset T$ defined over $L$), and let $\bar{\bf a} \subset V$ be the alcove in the opposite Weyl chamber whose closure contains our special vertex.

There are two decompositions of the Iwahori-Weyl group over $L$: $\widetilde{W}(G) = X_*(T)_I \rtimes W_0$, and $\widetilde{W}(G) = W_{\rm aff}(\Sigma) \rtimes \Omega$, where $\Omega \subset \widetilde{W}(G)$ is the stabilizer of the alcove $\bar{\bf a}$ for the natural action of $\widetilde{W}(G)$ on $V$. The torsion subgroup $X_*(T)_{I, \rm tors} \subset X_*(T)_I$ lies in $\Omega$, and in fact belongs to the center of $\widetilde{W}(G)$. This follows from the fact that $W_0$ acts trivially on $X_*(T)_{I, \rm tors}$: if $w \in W_0$ and $\lambda \in X_*(T)_{I, \rm tors}$, then $w\lambda - \lambda \in Q^\vee(\Sigma) \cap X_*(T)_{I,\rm tors} = 0$, where $Q^\vee(\Sigma)$ is the coroot lattice for  $\Sigma$. (We are using \cite[Lemma 15]{HR1}, which identifies $Q^\vee(\Sigma)$ with the subgroup $X_*(T_{\rm sc})_I$ of $X_*(T)_I$ and thus shows that $Q^\vee(\Sigma)$ is torsion-free, because $G_{\rm sc}$ is quasi-split over $L$ by a theorem of Steinberg and thus $X_*(T_{\rm sc})$ has a basis permuted by $I$.)

We write a superscript ``$\flat$'' to designate the result of quotienting by the torsion group $X_*(T)_{I, \rm tors}$, so that we have decompositions $\widetilde{W}(G)^\flat = X_*(T)^\flat_I \rtimes W_0 = W_{\rm aff}(\Sigma) \rtimes \Omega^\flat$. The root datum $\Sigma = (X^*, X_*, R, R^\vee)$ has $X_* = X_*(T)^\flat_{I}$ and extended affine Weyl group $\widetilde{W}(G)^\flat$. Let $x^\flat$ denote the image of $x \in \widetilde{W}(G)$ under $\widetilde{W}(G) \twoheadrightarrow \widetilde{W}(G)^\flat$.

\subsection{Proof of the vertexwise admissibility conjecture}
 Let $\{ \mu \}$ be the $G$-conjugacy class of a 1-parameter subgroup $\mu \in X_*(G)$; we may view this as an element of $X_*(T)/ W$. Following \cite[$\S4.3$]{PRS}, let $\widetilde{\Lambda}_{\{\mu\}} \subset \{ \mu \}$ be the subset of characters which are $B$-dominant for some Borel subgroup $B \supset T$ which is defined over $L$; this is a single $W_0$-conjugacy class. Let $\Lambda_{\{\mu\}}$ be the image of $\widetilde{\Lambda}_{\{\mu\}}$ in $X_*(T)_I$. Write $W_J \subseteq W_{\rm aff}(\Sigma)$ for the subgroup generated by a subset $J \subseteq S_{\rm aff}(\Sigma)$ of simple affine reflections. Then by definition ${\rm Adm}^J(\{\mu\}) = W_J {\rm Adm}(\{\mu\}) W_J$ where
$$
{\rm Adm}(\{\mu \}) := \{  x \in \widetilde{W}(G) ~ | ~ x \leq t_{\lambda}, \,\,\, \mbox{for some $\lambda \in \Lambda_{\{\mu \}}$} \};
$$
here $\leq$ denotes the Bruhat order defined via the decomposition $\widetilde{W}(G) = W_{\rm aff}(\Sigma) \rtimes \Omega$. 

The proof of the following lemma is left to the reader.

\begin{lemma} \label{flat_lem}
\begin{enumerate}
\item[(i)] Fix $\tau \in \Omega$ and suppose $x, y \in W_{\rm aff}(\Sigma)\tau \subset \widetilde{W}(G)$. Then $x \leq y \Leftrightarrow x^\flat \leq y^\flat$.
\item[(ii)] Suppose $\Lambda_{\{\mu \}} \subset W_{\rm aff}(\Sigma)\tau$ for $\tau \in \Omega$. Then for each subset $J$ as above ${\rm Adm}^J(\{\mu \})$ is the preimage of ${\rm Adm}^J(\Lambda^\flat_{\{\mu\}})$ under the map $W_{\rm aff}(\Sigma)\tau \rightarrow \widetilde{W}(G)^\flat$. Here ${\rm Adm}^J(\Lambda^\flat_{\{\mu\}})$ is the admissible set associated to the reduced root datum $\Sigma = (X^*, X_*, R, R^\vee)$ and $\Lambda^\flat_{\{\mu\}} \subset X_* =  X_*(T)^\flat_I$.
\end{enumerate}
\end{lemma}

Now Theorem \ref{PRS_conj} may be reformulated as follows.

\begin{theorem} \label{PRS_conjfinal}
Suppose $\mathcal J$ is a collection of subsets $J \subseteq S_{\rm aff}(\Sigma)$, and let $K = \bigcap_{J \in \mathcal J} J$.  Then
$$
\Adm^K(\{\mu\})=\bigcap_{J \in \mathcal J} \Adm^J(\{\mu\}).
$$
\end{theorem}

\begin{proof}
Using Lemma \ref{flat_lem}, the desired formula follows by applying the inverse image of $W_{\rm aff}(\Sigma) \tau \rightarrow \widetilde{W}(G)^\flat$ to the formula
$$
\Adm^K(\Lambda^\flat_{\{\mu \}}) = \bigcap_{J \in \mathcal J} \Adm^J(\Lambda^\flat_{\{\mu\}})
$$
proved in Theorem \ref{PRS_prelimconj}.
\end{proof}

\begin{Remark} \label{final_rem} Theorem \ref{PRS_conj} has implications for local models. For example, as in \cite[$\S4$]{PR} consider the morphism of naive local models associated to $(G, \{\mu \})$ where $G$ is a general unitary group
$$
\pi_{K,i} : M^{\rm naive}_K \rightarrow M^{\rm naive}_{\{i\}}
$$
for $i \in K$.  We take the point of view that defines the ``true'' local model $M^{\rm loc}_K$ as the scheme-theoretic closure of the generic fiber of $M^{\rm naive}_K$ in the scheme $M^{\rm naive}_K$. Then the coherence conjecture (now a theorem due to X.~Zhu \cite{Zhu}) together with Theorem \ref{PRS_conj} implies that we have an equality of schemes
$$
M^{\rm loc}_K = \bigcap_{i \in K} \pi^{-1}_{K, \{i\}}\big( M^{\rm loc}_{\{i\}} \big)
$$
(the intersection taken in $M^{\rm naive}_K$).  This shows that the local model attached to any facet can be recovered from local models attached to {\em vertices}. The proof is given in \cite[Prop.~4.5]{PR}.  
\end{Remark}

\bigskip

\noindent {\em Acknowledgements.} This work was written during the Fall 2014 MSRI program on Geometric Representation Theory. The first author thanks MSRI for hosting the program and for providing financial support for his visit. It is a pleasure to thank Michael Rapoport and Robert Kottwitz for several stimulating conversations about these questions.


\begin{thebibliography}{Ko90b}





\bibitem[Deo]{Deo} V.~Deodhar, {\em Some characterizations of Bruhat ordering on a Coxeter group and determinantion of the relative M\"{o}bius function}, Invent.~Math.~{\bf 39} \, (1977), 187-198.

\bibitem[Go1]{Go1} U. G\"{o}rtz, {\em On the flatness of models of certain Shimura varieties of PEL-type}, Math. Ann. {\bf 321} (2001), no. 3, 689-727.

\bibitem[Go2]{Go2} U. G\"{o}rtz, {\em On the flatness of local models for the symplectic group}, Adv. Math. {\bf 176}, (2003), no.1, 89-115.

\bibitem[Go3]{Go3} U.~G\"{o}rtz, {\em Topological flatness of local models in the ramified case},
Math.~Z.~{\bf 250} (2005), 775-790.

\bibitem[H01]{H01}
T.~Haines, {\em Test functions for Shimura varieties: the Drinfeld case}, Duke Math. J. {\bf 106} (2001), no.~1, 19-40.

\bibitem[HN]{HN} T.~Haines, B.~C.~Ng\^{o}, {\em Alcoves associated to special fibers of local models}, Amer. J. Math. {\bf 124} (2002), 1125-1152.


\bibitem[HR1]{HR1} T.~Haines, M.~Rapoport, {\em On parahoric subgroups}, Advances in Math. {\bf 219}  (2008), 188-198; appendix to: G.~Pappas, M.~Rapoport, {\em Twisted loop groups and their affine flag varieties}, Advances in Math. {\bf 219} (2008), 118-198.

\bibitem[HR2]{HR2}  T.~Haines, M.~Rapoport, {\it Shimura varieties with $\Gamma_1(p)$-level via Hecke algebra isomorphisms: the Drinfeld case}, Ann. Scient. \'Ecole  Norm. Sup. $4^e$ s\'{e}rie, t.~{\bf 45}, (2012), 719-785.


\bibitem[HP]{HP} T. Haines, A. Pettet, {\em Formulae relating the Bernstein and 
Iwahori-Matsumoto presentations of an affine Hecke algebra}, J. of Alg.  {\bf 252} (2002), 127-149.

\bibitem[He]{He} X.~He, {\em Kottwitz-Rapoport conjecture on unions of affine Deligne-Lusztig varieties}, arXiv:1408.5838, to appear in Ann. Sci. \`Ecole Norm. Sup.

\bibitem[HL]{HL} X.~He, Th.~Lam, {\em Projected Richardson varieties and affine Schubert varieties}, arXiv:1106.2586, to appear in Annales de l'Institut Fourier.

\bibitem[Hum]{Hum} J.~E.~Humphreys, {\em Reflection groups and Coxeter groups}, Cam.~Stud.~Adv.~Math.~{\bf 29}, 1990. 204 pp.+xii. 

\bibitem[KR]{KR} R. Kottwitz, M. Rapoport, {\em Minuscule alcoves for $GL_n$ and $GSP_{2n}$}, Manuscripta Math. {\bf 102}, no.4, (2000), 403-428.


\bibitem[PR]{PR}
G.~Pappas, M.~Rapoport, {\em Local models in the ramified case. III. Unitary groups}, J. Inst. Math. Jussieu, {\bf 8} (2009), 507--564.

\bibitem[PRS]{PRS} G.~Pappas, M.~Rapoport, B.~Smithling, {\em Local models of Shimura varieties, I.~Geometry and combinatorics}. Handbook of moduli.~Vol.~III, 135-217, Adv.~Lect.~Math.~(ALM), {\bf 26}, Int.~Press, Somerville, MA, 2013. 

\bibitem[PZ]{PZ} 
G.~Pappas, X.~Zhu, {\em Local models of Shimura varieties and a conjecture of Kottwitz}, Invent.~Math.~{\bf 194}, (2013), no.1, 147-254.


\bibitem [Ra05] {Ra05} M.~Rapoport: A guide to the reduction modulo $p$ of Shimura varieties.  Ast\'erisque {\bf 298}, (2005), 271-318.

\bibitem[RZ]{RZ} M.~Rapoport, T.~Zink, {\em Period spaces for 
$p$-divisible groups}, Annals of Math. Studies {\bf 141}, 
Princeton University Press 1996.



\bibitem[Sau]{Sau} L.~Sauermann, {\em On the $\mu$-admissible set in the extended affine Weyl groups of ${\bf E}_6$ and ${\bf E}_7$}, preprint 2014.



\bibitem[Sm1]{Sm2} B.~Smithling, {\em Admissibility and permissibility for minuscule cocharacters in orthogonal groups}, Manuscripta Math.~{\bf 136}, no.3-4, (2011), 295-314.

\bibitem[Sm2]{Sm3} B.~Smithling, {\em Topological flatness of local models for ramified unitary groups. I. The odd dimensional case}, Adv.~Math.~{\bf 226}, no.4, (2011), 3160-3190.

\bibitem[Sm3]{Sm4} B.~Smithling, {\em Topological flatness of local models for ramified unitary groups. II. The even dimensional case}, J.~Inst.~Math.~Jussieu~{\bf 13} (2014), no.~2, 303-393. 

\bibitem[Sm4]{Sm1} B.~Smithling, {\em Topological flatness of orthogonal local models in the split, even case}, I. Math.~Ann.~{\bf 35}, no.2, (2011), 381-416.




\bibitem[Zhu]{Zhu} X.~Zhu, {\em On the coherence conjecture of Pappas and Rapoport}, Annals Math.~{\bf 180} (2014), 1-85.

\end{thebibliography}
\end{document}